\documentclass[11pt,a4paper]{amsart}
\usepackage{amssymb}
\usepackage{latexsym}
\usepackage{exscale}
\usepackage{amsfonts}
\usepackage{graphicx}
\usepackage{mathrsfs}
\usepackage{amsmath,amscd,amsthm}
\usepackage{bbm,pifont}
\usepackage{enumerate}
\usepackage{color}
\usepackage{amssymb}
\usepackage{latexsym}
\usepackage{exscale}

\allowdisplaybreaks

\DeclareMathAlphabet\mathcalbf{OMS}{cmsy}{b}{n}
\DeclareMathAlphabet\EuScript{U}{eus}{m}{n}
\DeclareMathAlphabet\EuScriptBold{U}{eus}{b}{n}

\numberwithin{equation}{section}
\newtheorem{theorem}{Theorem}[section]
\newtheorem{lemma}[theorem]{Lemma}
\newtheorem{corollary}[theorem]{Corollary}
\newtheorem{proposition}[theorem]{Proposition}

\newtheorem{example}[theorem]{Example}
\newtheorem{remark}[theorem]{Remark}

\newcommand{\diam}{\operatorname{diam}}

\def\R{\mathbb R}

\def\C{\mathbb C}

\begin{document}
\allowdisplaybreaks

\title[Commutators of maximal functions on spaces of homogeneous type]
{Commutators of maximal functions on spaces of homogeneous type and their weighted, local versions}
\author{ Zunwei Fu,  Elodie Pozzi and Qingyan Wu}

\address{Zunwei Fu, Department of Mathematics\\
         Linyi University\\
         Shandong, 276005, China
         }
\email{fuzunwei@lyu.edu.cn}

\address{Elodie Pozzi, Department of Mathematics and Statistics\\
         Saint Louis University\\
         220 N. Grand Blvd, 63103 St Louis MO, USA
         }
\email{elodie.pozzi@slu.edu}

\address{Qingyan Wu, Department of Mathematics\\
         Linyi University\\
         Shandong, 276005, China
         }
\email{wuqingyan@lyu.edu.cn}


\date{\today}



\begin{abstract}
We establish the characterizations of commutators of several versions of
maximal functions on spaces of homogeneous type. In addition, with the aid of interpolation theory, we provide weighted version of the commutator theorems by establishing new characterizations of the weighted BMO space. Finally, a concrete example shows the local version of commutators also has an independent interest.
\end{abstract}

\maketitle




\section{Introduction}
\label{sec:introduction}
\setcounter{equation}{0}

On the Euclidean space, for the Hilbert transform $H$, and other classical singular integral operators, a well-known result due to Coifman, Rochberg and Weiss (cf. \cite{CRW}) states that a locally integrable function $b$ in $\mathbb R^n$ is in ${\rm BMO}$ if and only if the commutator $[H, b]f:=H(bf)-bH(f)$, is bounded in $L^p$, for some (and for all) $1<p<\infty$. In \cite[Propositions 4 and 6]{BMR},  Bastero, Milman and Ruiz characterized the class of functions for which the commutator with the Hardy-Littlewood maximal function and  the maximal sharp function are bounded on $L^p$. Later, Garc\'ia-Cuerva et al. \cite[Theorem 2.4]{GHST} proved that the maximal commutator $\mathcal{C}_b$ is bounded on $L^p(\mathbb R^n)$, $1<p<\infty$,
 if and only if $b\in {\rm BMO}(\mathbb R^n)$,
where $\mathcal{C}_b$ is defined by
$$\mathcal{C}_b(f)(x)=\sup_{Q\ni x}{1\over |Q|}\int_Q|b(x)-b(y)| |f(y)|dy.$$

For the endpoint case, Agcayazi, Gogatishvili, Koca and Mustafayev \cite[Theorems 1.5, 1.6 and 1.8]{AGKM} proved that the commutator $[\mathcal{M}, b]$ of the Hardy--Littlewood maximal function $\mathcal{M}$ is bounded from $L\log^+L(\mathbb R^n)$ into weak $L^1(\mathbb R^n)$ if and only if $b$ is in ${\rm BMO}(\mathbb R^n)$. Similar result holds for the maximal commutator $\mathcal{C}_b$. Regrading weak type $(1,1)$, we note that in \cite{AGKM}, the authors also gave a counterexample to show that $[\mathcal{M}, b]$ fails to be of weak type $(1,1)$.

In this paper, we want to extend the above results to the spaces of homogeneous type. In modern harmonic analysis, it has been a central theme to extend the real variable theory from the Euclidean setting, namely where the underlying space is $\R^n$ with the Euclidean metric and Lebesgue measure, to more general
settings. To this end, Coifman and Weiss formulated the concept of spaces of homogeneous type, in~\cite{CW1}. There is a large literature devoted to spaces of homogeneous type, see for example~\cite{Aimar2007, AuHyt, CLW, CW1,CW2,DJS,DH,H1,H2,HLL,HS,MS}. Some non-Euclidean examples of spaces of homogeneous type are given by the Carnot--Carath\'eodory spaces whose theory is developed by Nagel, Stein and others
in~\cite{NS1}, \cite{NS2} and related papers; there the quasi-metric is defined in terms of vector fields satisfying the H\"ormander condition on an underlying manifold. Very recently, He et. al \cite{HHLLYY} presented a complete real-variable theory of Hardy spaces on
spaces of homogeneous type without additional (geometrical) measure condition in which only {\it doubling condition} is required.

Recall that $(X,\rho,\mu)$ is a {\it space of homogeneous type} in the sense of Coifman and Weiss if $\rho$~is a
quasi-metric and $\mu$ is a nonzero measure satisfying the doubling condition. A
\emph{quasi-metric}~$\rho: X\times X\rightarrow[0,\infty)$
satisfies (i) $\rho(x,y) = \rho(y,x) \geq 0$ for all $x$,
$y\in X$; (ii) $\rho(x,y) = 0$ if and only if $x = y$; and
(iii) the \emph{quasi-triangle inequality}: there is a constant
$A_0\in [1,\infty)$ such that for all $x$, $y$, $z\in X$,
\begin{eqnarray}\label{eqn:quasitriangleineq}
    \rho(x,y)
    \leq A_0 [\rho(x,z) + \rho(z,y)].
\end{eqnarray}
In contrast to a metric, the quasi-metric may not be H\"older
regular and quasi-metric balls may not be open; see for
example~\cite[p.5]{HK}. A nonzero measure $\mu$ satisfies the
\emph{doubling condition} if there is a constant $C_\mu$ such
that for all $x\in X$ and all $r > 0$,
\begin{eqnarray}\label{doubling condition}
   \mu(B(x,2r))
   \leq C_\mu \mu(B(x,r))
   < \infty,
\end{eqnarray}
where $B(x,r)=\{y\in X: \rho(x,y)<r\}$ with $x\in X$ and $r>0$ are the $\rho$-balls.
We point out that the doubling condition (\ref{doubling
condition}) implies that there exists a positive constant
$n$ (the \emph{upper dimension} of~$\mu$)  such
that for all $x\in X$, $\lambda\geq 1$ and $r > 0$,
\begin{eqnarray}\label{upper dimension}
    \mu(B(x, \lambda r))
    \leq  C_\mu\lambda^{n} \mu(B(x,r)).
\end{eqnarray}
Throughout this paper, we assume that $\mu$ satisfies the following conditions:
$\mu(\{x_0\})=0$ for every $x_0\in X$.

We now consider the case $\mu(X)=\infty$ (See also the case $\mu(X)<\infty$ in Section 4). Note that ${\rm BMO }(X)$ is defined as the set of all $b\in L^1_{loc}(X)$ such that
$$ \sup_{ B\subset X} {1\over\mu(B)}\int_{B} |b(y)-m_B(b)|d\mu(y)<\infty, $$
where $m_B(f)$ is the average of $f$ on $B$, i.e.
$m_B(f)={1\over \mu(B)}\int_B f(y)d\mu(y)$, and the norm is defined as
$$\|b\|_{{\rm BMO }(X)}:= \sup_{ B\subset X} {1\over\mu(B)}\int_{B} |b(y)-m_B(b)|d\mu(y). $$
 Note that it also has an {equivalent norm}, defined by
 $$\|b\|'_{\operatorname{BMO}(X)}=\sup_{B\subset X}\inf_{c\in\mathbb R}\frac{1}{\mu(B)}\int_{B}|b(y)-c|d\mu(y).$$

For a locally integrable function $f$ on $X$ and for $1\leq p<\infty$, the {\it Hardy--Littlewood maximal function} $M_pf$ of $f$ is defined by
$$M_p(f)(x)=\sup_{B\ni x}\left({1\over \mu(B)}\int_B|f(y)|^pd\mu(y)\right)^{1\over p},$$
where the supremum is taken over all balls containing $x$. When $p=1$, we write $M=M_1$, which is the classical Hardy--Littlewood maximal operator. For any $f\in L^1_{loc}(X)$ and $x\in X$, let $M^\sharp f$ be the {\it sharp maximal function of Fefferman and Stein} defined by
$$M^\sharp f(x)=\sup_{B\ni x}{1\over \mu(B)}\int_B |f(y)-m_B(f)|d\mu(y).$$
For a fixed $\delta\in (0,1)$, any suitable function $f$ and $x\in X$, let
$$M^\sharp_{\delta}f(x):=\left[M^\sharp\left(|f|^\delta\right)(x)\right]^{1\over \delta},
\quad
M_{\delta}f(x):=\left[M\left(|f|^\delta\right)(x)\right]^{1\over \delta}.$$

\noindent Given a measurable function $b$, the commutator of the Hardy--Littlewood maximal operator $M_{p}$ and $b$ is defined by
$$[M_{p}, b]f(x):=M_{p}(bf)(x)-b(x)M_{p}f(x) $$
for all $x\in X$. $[M^{\sharp}, b]$ can be defined in the same way.
 The maximal commutator is defined by
$$C_b(f)(x):=\sup_{B\ni x}{1\over\mu(B)}\int_{B}|b(x)-b(y)| |f(y)|d\mu(y)$$
for all $x\in X$. For any ball $B\subset X$, define
$$M_{p,B}(f)(x)=\sup_{x\in {B_0}\subseteq B}\Big({1\over \mu(B_0)}\int_{B_0}|f(y)|^pd\mu(y)\Big)^{1\over p}.$$

This paper is organised as follows. In Section 2 we recall some necessary
definitions and results on space of homogeneous type. We mainly present the pointwise estimate for $[M_{p}, b]$, $C_b$ and $[M^\sharp, b]$. We also provide a counter example on space of homogeneous type showing that the commutator  $[M,b]$ fails to be of weak type $(1,1)$.
Section 3 focuses on the weighted version of commutator theorems.
In the last section we skim through the local characterizations and give a concrete example on bounded pseudoconvex domain in $\mathbb C^n$ which shows
the local version of commutators also has its independent connotation.

%
%






\medskip
\section{Pointwise estimates}
\label{sec:pointwise}
\setcounter{equation}{0}

The poinwise estimates for the characterizations of commutators of maximal functions on spaces of homogeneous type depend on the case of the Euclidean space appeared in \cite{AGKM}. However, this is the basis for later weighted and local estimations. Therefore, we will still give a complete proof with the aid of proof based on [1] to keep the integrity of the paper.

\begin{theorem}\label{thm-mp}
Let $b$ be a real valued, locally integrable function in $X$. The following assertions are equivalent:

{\rm (i)} $b\in {\rm BMO}(X)$ and $b^-\in L^\infty(X)$;

{\rm (ii)} The commutator $[M_p, b]$ is bounded on $L^q(X)$, for all $q$, $p<q<\infty$;

{\rm (iii)} The commutator $[M_p, b]$ is bounded on $L^q(X)$, for some $q$, $p<q<\infty$;

{\rm (iv)} For all $q\in [1,\infty)$ we have
$$\sup_{B}{1\over\mu(B)}\int_B|b(x)-M_{p,B}(b)(x)|^qd\mu(x)<\infty.$$

{\rm (v)} There exists $q\in [1,\infty)$ such that
$$\sup_{B}{1\over\mu(B)}\int_B|b(x)-M_{p,B}(b)(x)|^qd\mu(x)<\infty.$$
\end{theorem}

\begin{theorem}\label{thm-msharp}
Let $b$ be a real valued, locally integrable function in $X$. The following assertions are equivalent:

{\rm (i)} $b\in {\rm BMO}(X)$ and $b^-\in L^\infty(X)$;

{\rm (ii)} The commutator $[M^\sharp, b]$ is bounded on $L^q(X,\mu)$, for all $q$, $1<q<\infty$;

{\rm (iii)} The commutator $[M^\sharp, b]$ is bounded on $L^q(X,\mu)$, for some $q$, $1<q<\infty$;

{\rm (iv)} For all $q\in [1,\infty)$, we have
$$\sup_{B}{1\over\mu(B)}\int_B|b(x)-2M^\sharp(b\chi_B)(x)|^qd\mu(x)<\infty;$$

{\rm (v)} There exists $q\in [1,\infty)$ such that
$$\sup_{B}{1\over\mu(B)}\int_B|b(x)-2M^\sharp(b\chi_B)(x)|^qd\mu(x)<\infty.$$
\end{theorem}

\begin{theorem}\label{thm-mb}
Let $b\in{\rm BMO}(X)$ such that $b^-\in L^\infty(X)$. Then there exists a positive constant $C$ such that
for all $f\in L(1+\log^+L)(X)$ and $\lambda>0$,
\begin{align*}
&\mu\left(\{x\in X: |[M,b]f(x)|>\lambda \}\right)\\
&\leq C\left(\|b^+\|_{{\rm BMO}(X)}+\|b^-\|_{L^\infty} \right) \big(1+\log^+\left(\|b^+\|_{{\rm BMO}(X)}\right)\big)
\int_X{|f(x)|\over \lambda}\Big(1+\log^+\Big( {|f(x)|\over \lambda}\Big) \Big)d\mu(x).
\end{align*}
\end{theorem}

\begin{theorem}\label{thm-lp}
Let $b\in L_{loc}^1(X)$ and $1<p<\infty$. Then the maximal commutator $C_b$ is bounded on $L^p(X)$ if and only if $b\in {\rm BMO(X)}$.
\end{theorem}

\begin{theorem}\label{thm-cb-iff}
Let $b\in L^1_{loc}(X)$. Then $b\in {\rm BMO}(X)$ if and only if there exists a positive constant $C$ such that for each $\lambda>0$ and all $f\in L(1+\log^+L)(X)$,
\begin{align}\label{cbl}
\mu\left(\left\{ x\in X: C_b(f)(x)>\lambda \right\} \right)\leq C\int_X{|f(x)|\over\lambda}\Big(1+\log^+\Big({|f(x)|\over\lambda} \Big) \Big)d\mu(x).
\end{align}
\end{theorem}

We mention here the {\it sufficiency} of $b\in {\rm BMO}(X)$ for the boundedness of $C_b$ in Theorem \ref {thm-lp} and Theorem \ref {thm-cb-iff} has been studied by Hu, Lin and Yang (see \cite{HLY}) in a different way. We will give the proof of the converse part.

\indent For any locally integrable function $f$, let $f^*_\mu$ be the non-increasing rearrangement  of $f$ (see for example \cite{Le}), namely,
$$f^*_\mu(t)=\inf\{s>0:\mu(\{x\in X: |f(x)|>s \})<t \},\quad 0<t<\infty.$$


The following John-Nirenberg inequalities on spaces of homogeneous type come from \cite[Propositions 6, 7]{Kr}.

\begin{lemma}\label{lem-jn1}
If $f\in {\rm BMO}(X)$, then there exist positive constants $C_1$ and $C_2$ such that for every ball $B\subset X$ and every $\alpha>0$, we have
$$\mu(\{x\in B: |f(x)-m_B(f)|>\alpha \})\leq C_1\mu(B)\exp\Big\{- {C_2\over \|f\|_{{\rm BMO}(X)}}\alpha\Big\}.$$
\end{lemma}

\begin{lemma}\label{lem-jn2}
If $f\in {\rm BMO}(X)$, then there exist positive constants $C_3$ and $C_4$ such that for every ball $B\subset X$,
$$\int_{B}\exp\{C_3|f(x)-m_B(f)| \}d\mu(x)<C_4\mu(B).$$
\end{lemma}

On spaces of homogeneous type, we also have the following equivalent ${\rm BMO}$ norm, see for example \cite[Theorem 5.5]{HT}.

\begin{lemma}\label{lem-bmoq}
Let $0<p<\infty$, and $f$ be a measurable function on $X$. Then $f$ is in $L_{loc}^1(X)$ and satisfies
$$\sup_{B}\left( {1\over \mu(B)}\int_B|f(x)-m_B(f)|^pd\mu(x)\right)^{1\over p}<\infty$$
if and only if $f$ is in ${\rm BMO}(X)$. In such a case,
we have
$$ \|f\|_{{\rm BMO}(X)}\approx\sup_{B}\left( {1\over \mu(B)}\int_B|f(x)-m_B(f)|^pd\mu(x)\right)^{1\over p}.$$
\end{lemma}
\setcounter{equation}{0}

A generalized H\"older's inequality will be used in our argument.
 For any measurable set $E$ with $\mu(E)<\infty$, for any suitable function $f$, the norm $\|f\|_{L\log L, E}$ is defined by
\begin{align*}
\|f\|_{L\log L, E}=\inf\left\{\lambda>0: {1\over \mu(E)}\int_E{|f(y)|\over \lambda}\log\Big(2+{|f(y)|\over\lambda}\Big)d\mu(y)\leq 1 \right\}.
\end{align*}
The maximal operator $M_{L\log L}$ is defined by
$$M_{L\log L}f(x)=\sup_{B\ni x}\|f\|_{L\log L, B}.$$
The norm $\|f\|_{\exp L, E}$ is defined by
$$\|f\|_{\exp L, E}=\inf\left\{\lambda>0:{1\over \mu(E)}\int_E\exp\left({|f(x)|\over\lambda}\right)d\mu(x) \leq 2  \right\}.$$
Then the following generalized H\"older's inequality:
\begin{align}\label{holder}
{1\over \mu(E)}\int_E |f(x)g(x)|d\mu(x)\leq C\|f\|_{L\log L, E}\|g\|_{\exp L, E}
\end{align}
holds for any suitable functions $f$ and $g$ (see for example \cite{RR}). From \cite[p. 90]{PS}, we can see
\begin{align}\label{mml}
M^2f(x)\approx M_{L\log L}f(x),
\end{align}
where $M^2=M\circ M$.


\begin{lemma}[\cite{CS}]\label{lem-ml}
There is a positive constant $C$ such that for any bounded function $f$ with bounded support and for all $\lambda>0$,
$$\mu\left(\{y\in X: M_{L\log L}f(x)>\lambda \}\right)\leq C\int_X {|f(y)|\over \lambda}\left(1+\log^+\left({|f(y)|\over \lambda} \right) \right)d\mu(y).$$
\end{lemma}

We need the following relation between operators $C_b$ and $[M,b]$. The proofs are similar to those of Lemma 3.1 and Lemma 3.2 in \cite{AGKM}, respectively.

\begin{lemma}\label{mc}
Let $b$ be any non-negative locally integrable function. Then
$$|[M, b]f(x)|\leq C_b(f)(x)$$
for all $f\in L_{loc}^1(X)$.
\end{lemma}


\begin{lemma}\label{lem-cbm}
Let $b$ be any locally integrable function on $X$. Then
$$|[M,b]f(x)|\leq C_b(f)(x)+2b^-(x)Mf(x)$$
holds for all $f\in L_{loc}^1(X)$.
\end{lemma}

\begin{proposition}\label{thm-mm}
Let $b\in {\rm BMO}(X)$ and let $0<\delta<1$. Then there exists a positive constant $C=C(\delta)$ such that
$$M_\delta(C_b(f))(x)\leq C\|b\|_{{\rm BMO}(X)}M^2f(x),\quad x\in X,$$
for all $f\in L_{loc}^1(X)$.
\end{proposition}

\begin{proof}[Proof of Proposition \ref {thm-mm}]
Let $x\in X$ and fix a ball $B$ containing $x$. Let $f=f_1+f_2$, where $f_1=f\chi_{3B}$.  For any $y\in X$, we have
\begin{align*}
C_b(f)(y)&=M\big((b-b(y))f\big)(y)=M\big((b-m_{3B}(b)+m_{3B}(b)-b(y))f\big)(y)\\
&\leq M((b-m_{3B}(b))f)(y)+|m_{3B}(b)-b(y)|Mf(y)\\
&\leq M((b-m_{3B}(b))f_1)(y)+M((b-m_{3B}(b))f_2)(y)+|m_{3B}(b)-b(y)|Mf(y).
\end{align*}
Therefore,
\begin{align*}
\left({1\over \mu(B)}\int_B (C_b(f)(y))^\delta d\mu(y) \right)^{1\over \delta}
&\leq
\left({1\over \mu(B)}\int_B \left| M((b-m_{3B}(b))f_1)(y)\right|^\delta d\mu(y) \right)^{1\over \delta}\\
&\quad+
\left({1\over \mu(B)}\int_B \left| M((b-m_{3B}(b))f_2)(y)\right|^\delta d\mu(y) \right)^{1\over \delta}\\
&\quad+
\left({1\over \mu(B)}\int_B |m_{3B}(b)-b(y)|^\delta (Mf(y))^\delta d\mu(y) \right)^{1\over \delta}\\
&=I+II+III.
\end{align*}

We first estimate $I$.  Recall that $M$ is weak-type $(1,1)$ (c.f. \cite[P. 299]{GLY}). We have
\begin{align*}
\int_B \left| M((b-m_{3B}(b))f_1)(y)\right|^\delta d\mu(y)
&\leq \int_{0}^{\mu(B)} \left[\left(M((b-m_{3B}(b))f_1) \right)^*(t) \right]^\delta dt\\
&\leq \left[\sup_{0<t<\mu(B)}t\left(M((b-m_{3B}(b))f_1) \right)^*(t) \right]^\delta
\int_{0}^{\mu(B)}t^{-\delta}dt\\
&\leq C \left\|(b-m_{3B}(b))f_1\right\|_{L^1(X)}^{\delta}\mu(B)^{-\delta+1}\\
&\leq C \left\|(b-m_{3B}(b))f\right\|_{L^1(3B)}^{\delta}\mu(B)^{-\delta+1}.
\end{align*}
Thus
$$I\leq C {1\over \mu(B)}\int_{3B}\left|b(y)-m_{3B}(b) \right| |f(y)|d\mu(y).$$
Then by \eqref{holder} and the John-Nirenberg inequality, we have
\begin{align*}
I&\leq C \|b-m_{3B}(b)\|_{\exp L, 3B}\|f\|_{L\log L, 3B}\\
&\leq C \|b\|_{{\rm BMO}(X)}M_{L\log L}f(x).
\end{align*}

For $II$, since for any two points $x, y\in B$, we have
$$M((b-m_{3B}(b))f)(y)\leq CM((b-m_{3B}(b))f)(x)$$
with $C$ an absolute constant  (see for example \cite[p. 160]{GR}).
Then by \eqref{holder}, we can get
\begin{align}\label{ii}
\nonumber II&\leq C M((b-m_{3B}(b))f)(x)\\
&=\sup_{B\ni x}{1\over \mu(B)}\int_B |b(y)-m_{3B}(b)| |f(y)|d\mu(y)\\
\nonumber&\leq C \sup_{B\ni x} \|b-m_{3B}(b)\|_{\exp L, 3B}\|f\|_{L\log L, 3B}\\
\nonumber&\leq C \|b\|_{{\rm BMO}(X)}M_{L\log L}f(x).
\end{align}

For $III$, by H\"older's inequality, Lemma \ref {lem-bmoq}, we have
\begin{align*}
III&= \left({1\over \mu(B)}\int_B |m_{3B}(b)-b(y)|^\delta (Mf(y))^\delta d\mu(y) \right)^{1\over \delta}\\
&\leq\left({1\over \mu(B)}\int_B |m_{3B}(b)-b(y)|^{\delta\over 1-\delta}d\mu(y)\right)^{1-\delta\over\delta} \left({1\over \mu(B)}\int_B(Mf(y)) d\mu(y) \right)\\
&\leq C \|b\|_{{\rm BMO}(X)}M^2f(x).
\end{align*}
Therefore, by \eqref{mml}, we have
\begin{align*}
\left({1\over \mu(B)}\int_B (C_b(f)(y))^\delta d\mu(y) \right)^{1\over \delta}
&\leq C\|b\|_{{\rm BMO}(X)}\left(M_{L\log L}f(x)+M^2f(x)\right)\\
&\leq C\|b\|_{{\rm BMO}(X)}M^2f(x).
\end{align*}
This finishes the proof of Proposition \ref {thm-mm}.
\end{proof}

By the Lebesgue differentiation theorem, we can get the following corollary.

\begin{corollary}\label{cor-cm}
Let $b\in{\rm BMO}(X)$. Then there exists a positive constant $C$ such that for all $f\in L^{1}_{loc}(X)$,
$$C_b(f)(x)\leq C\|b\|_{{\rm BMO}(X)}M^2f(x),\quad x\in X.$$
\end{corollary}


\begin{corollary}\label{cor-cb}
Let $b\in{\rm BMO}(X)$ such that $b^-\in L^\infty(X)$. Then there exists a positive constant $C$ such that for all $f\in L^{1}_{loc}(X)$,
\begin{align*}
\left|[M, b]f(x) \right|\leq C\left(\|b^+\|_{{\rm BMO}(X)}+\|b^-\|_{L^\infty} \right) M^2f(x).
\end{align*}
\end{corollary}

\begin{proof}
By Lemma \ref{lem-cbm}, Corollary \ref{cor-cm} and the fact that $f\leq Mf$, we have
\begin{align*}
|[M,b]f(x)|&\leq C_b(f)(x)+2b^-(x)Mf(x)\\
&\leq C\left( \|b\|_{{\rm BMO}(X)}M^2f(x)+b^-(x)Mf(x)\right)\\
&\leq C\left[ \left(\|b^+\|_{{\rm BMO}(X)}+\|b^-\|_{{\rm BMO}(X)}\right)M^2f(x)+\|b^-\|_{L^\infty(X)}M(Mf)(x)\right]\\
&\leq C\left(\|b^+\|_{{\rm BMO}(X)}+\|b^-\|_{L^\infty} \right) M^2f(x).
\end{align*}
This finishes the proof.
\end{proof}

According to \cite[Theorem 4.4]{MiS}, we have the following boundedness result for quasilinear operators $T$, which satisfy\\
(a) $Tf\geq 0$, for $f\in D(T)$;\\
(b) $T(\alpha f)=|\alpha|Tf$, for $\alpha\in\mathbb R$ and $f\in D(T)$;\\
(c) $|Tf-Tg|\leq T(f-g)$, for $f,g\in D(T)$,\\
where $D(T)$ is a suitable class of locally integrable functions.

\begin{lemma}\label{lem-mpb}
Let $b\in{\rm BMO}(X)$ be a nonnegative function. Suppose that $T$ is a quasilinear operator satisfying (a)-(c) and is bounded on $L^q(X)$, for some $1\leq q<\infty$. Then
$[T, b]$ is bounded on $L^q(X)$.
\end{lemma}

Since $\left|[T,b]f-[T,|b|]f \right|\leq 2\left(b^-T(f)+T(b^-f) \right)$, for general {\rm BMO} function $b$, the following result holds (see also \cite[Proposition 3]{BMR})

\begin{proposition}\label{pro-mpb}
Let $b\in{\rm BMO(X)}$ with $b^-\in L^\infty(X)$. Suppose that $T$ is a quasilinear operator satisfying (a)-(c) and is bounded on $L^q(X)$, for some $1\leq q<\infty$. Then
$[T, b]$ is bounded on $L^q(X)$.
\end{proposition}

\begin{proof}[Proof of Theorem \ref{thm-mp}]
It is clear that (ii) implies (iii), (iv) implies (v). For (i)$\Rightarrow$(ii), since $M_p$ satisfies (a)-(c) and is bounded on $L^q(X)$ for $p<q<\infty$, the result follows from Proposition \ref{pro-mpb}. For (iii)$\Rightarrow$(v), by assumption,  $[M_p, b]$ is bounded on $L^q(X)$ for some $p<q<\infty$, then for any fixed ball $B\subset X$, we have
\begin{align}\label{mpb1}
\left(\int_B\left| M_p(bf)-bM_p(f)\right|^qd\mu\right)^{1\over q}\leq \left\|[M_p, b]f \right\|_{L^q(X)}\leq C\|f\|_{L^q(X)}.
\end{align}

Now we choose $f=\chi_B\in L^q(X)$, then $M_{p}(\chi_B)=\chi_B$ and $M_{p}(b\chi_B)(x)=M_{p,B}(b)(x)$ for all $x\in B$, thus by \eqref{mpb1}, we have
$$\left(\int_B\left| b-M_{p,B}(b)\right|^qd\mu\right)^{1\over q}\leq C\mu(B)^{1\over q},$$
which implies (v).

For (v)$\Rightarrow$(i), let $B$ be a fixed ball. By H\"older's inequality, we have
$${1\over\mu(B)}\int_B|b-M_{p,B}(b)|d\mu\leq\left({1\over \mu(B)}\int_B\left|b-M_{p,B}(b)\right|^qd\mu \right)^{1\over q}\leq C.$$

\noindent Let $E_1=\{x\in B: b(x)\leq m_{B}(b)\}$ and $E_2=\{x\in B: b(x)> m_{B}(b)\}$. It is clear that
$$\int_{E_1}|b-m_{B}(b)|d\mu=\int_{E_2}|b-m_{B}(b)|d\mu.$$
Since for $x\in E_1$, $b(x)\leq m_{B}(b)\leq M_{p,B}(b)(x)$, we have
\begin{align*}
{1\over \mu(B)}\int_B|b-m_{B}(b)|d\mu&={2\over \mu(B)}\int_{E_1}|b-m_{B}(b)|d\mu\leq {2\over \mu(B)}\int_{E_1}|b-M_{p,B}(b)|d\mu\\
&\leq{2\over \mu(B)}\int_B|b-M_{p,B}(b)|d\mu\leq C.
\end{align*}
Therefore, $b\in{\rm BMO}(X)$. Next, we will show that $b^-\in L^{\infty}(X)$. Observe that $M_{p,B}(b)\geq|b|$ in $B$, therefore, in $B$, we have
$$0\leq b^-\leq M_{p,B}(b)-b^++b^-=M_{p,B}(b)-b.$$

Together with the assumption of (v), we can see that there exists a constant $C$ such that for any ball $B$, we have
$$m_B(b^-)\leq C.$$
Then the boundedness of $b^-$ follows from Lebesgue's differentiation theorem. The implication of (ii)$\Rightarrow$(iv) is similar to (iii)$\Rightarrow$(v). This ends the proof of Theorem \ref {thm-mp}.
\end{proof}

%
%
%
%
%
%
%
%
%
%

\begin{proof}[Proof of Theorem \ref {thm-msharp}]
Obviously, (ii)$\Rightarrow$(iii), and (iv)$\Rightarrow$(v). For (i)$\Rightarrow$(ii), since $M^\sharp f\leq 2Mf$ and $M^\sharp$ satisfies (a)-(c), the result follows from Proposition \ref{pro-mpb}. For (iii)$\Rightarrow$(v). Let $B$ be a fixed ball and $B_1$ be any other ball. Then
\begin{align*}
&{1\over \mu(B_1)}\int_{B_1}\left|\chi_B(x)-m_{B_1}(\chi_B) \right|d\mu(x)\\
&={1\over \mu(B_1)}\int_{B_1}\left|{1\over \mu(B_1)}\int_{B_1}\left(\chi_B(x)-\chi_B(y)\right)d\mu(y) \right|d\mu(x)\\
&={1\over \mu(B_1)^2}\left(\int_{B_1\setminus B}\int_{B_1\cap B}d\mu(y) d\mu(x)
+ \int_{B_1\cap B}\int_{B_1\setminus B}d\mu(y) d\mu(x) \right)\\
&={2\mu(B_1\setminus B)\mu(B_1\cap B)\over \mu(B_1)^2}\leq {1\over 2}.
\end{align*}
Therefore,
\begin{align*}
M^\sharp\left(\chi_B \right)(x)=\sup_{x\in B_1}{1\over\mu(B_1)}\int_{B_1}\left|\chi_B-m_{B_1}(\chi_B) \right|d\mu\leq{1\over 2}
\end{align*}

On the other hand, given $x\in B$, there always exists a ball $B_1\supset B$ such that $\mu(B_1)=2\mu(B)$. Therefore,
\begin{align*}
M^\sharp\left(\chi_B \right)(x)={1\over 2},\quad x\in B.
\end{align*}

\noindent By the assumption of (iii), we can see that
\begin{align*}
\left(\int_B\left| M^\sharp(bf)-bM^\sharp(f) \right|^qd\mu \right)^{1\over q}\leq \big\|\big[ M^\sharp, b\big]f \big\|_{L^q(X)}\leq \|[M^\sharp, b]\|_{L^q(X)\to L^q(X)}\|f\|_{L^q(X)}.
\end{align*}

\noindent Now we choose  $f=\chi_B\in L^q(X)$, then $M^\sharp(\chi_B)(x)={1\over 2}$ for $x\in B$. Therefore,
$$\left(\int_B\Big| {1\over 2} b(x)-M^\sharp(b\chi_B)(x)\Big|^qd\mu(x)\right)^{1\over q}<\|[M^\sharp, b]\|_{L^q(X)\to L^q(X)}\mu(B)^{1\over q},$$
which implies (iv). For (v)$\Rightarrow$(i), we first prove that
$$|m_{B}(b)|\leq 2M^\sharp\left( b\chi_B\right) (x), \quad x\in B.$$

In fact, let $x\in B$ and take $B_1\supset B$ satisfying $\mu(B_1)=2\mu(B)$. Then
\begin{align*}
M^\sharp\left( b\chi_B\right)(x)&\geq {1\over \mu(B_1)}\int_{B_1}\left| b (y)\chi_B(y)-m_{B_1}(b\chi_B)\right|d\mu(y)\\
&={1\over \mu(B_1)}\left[\int_{B}\left| b (y)\chi_B(y)-m_{B_1}(b\chi_B)\right|d\mu(y)\right . \\
&\qquad+\left.\int_{B_1\setminus B}\left| b (y)\chi_B(y)-m_{B_1}(b\chi_B)\right|d\mu(y)\right]\\
&={1\over 2\mu(B)}\Big[\int_{B}\left| b (y)-m_{B_1}(b\chi_B)\right|d\mu(y)
+\left|m_{B_1}(b\chi_B)\right|\mu(B_1\setminus B)\Big].
\end{align*}

\noindent By definition,
\begin{align*}
m_{B_1}(b\chi_B)&={1\over \mu(B_1)}\int_{B_1}b(y)\chi_B(y)d\mu(y)\\
&={1\over 2\mu(B)}\Big[\int_{B_1\setminus B}b(y)\chi_B(y)d\mu(y)
+\int_{ B}b(y)\chi_B(y)d\mu(y)\Big]={1\over 2}m_{B}(b).
\end{align*}

\noindent Therefore, we have
\begin{align*}
M^\sharp\left( b\chi_B\right)(x)
&\geq {1\over 2\mu(B)}\Big[\int_{B}\big| b(y)-{1\over 2}m_{B}(b)\big|d\mu(y)
+{1\over 2}\left|m_{B}(b)\right|\mu(B_1\setminus B)\Big]\\
&={1\over 2\mu(B)}\int_{B}\big| b(y)-{1\over 2}m_{B}(b)\big|d\mu(y)
+{1\over 4}\left|m_{B}(b)\right|.
\end{align*}

On the other hand,
\begin{align*}
\left|m_{B}(b)\right|\leq{1\over\mu(B)}\int_B\left|b(y)-{1\over 2}m_{B}(b) \right| d\mu(y)+{1\over 2}|m_{B}(b)|.
\end{align*}
Thus
\begin{align*}
\left|m_{B}(b)\right|\leq 2 M^\sharp\left( b\chi_B\right)(x).
\end{align*}

Now we show that $b\in {\rm BMO}(X)$. Indeed, let
$E=\{x\in B:b(x)\leq m_{B}(b)\}$. Then
\begin{align*}
{1\over \mu(B)}\int_B\left| b(x)-m_{B}(b)\right|d\mu(x)
&={2\over \mu(B)}\int_E\left( b(x)-m_{B}(b)\right)d\mu(x)\\
&\leq{2\over \mu(B)}\int_E\left (2 M^\sharp\left( b\chi_B\right)(x)-b(x)\right)d\mu(x)\\
&={2\over \mu(B)}\int_E\left| 2 M^\sharp\left( b\chi_B\right)(x)-b(x)\right|d\mu(x)\\
&\leq{2\over \mu(B)}\int_B\left| 2 M^\sharp\left( b\chi_B\right)(x)-b(x)\right|d\mu(x)\leq C.
\end{align*}

Next we prove that $b^-\in L^\infty(X)$. Using the inequality

$$2M^\sharp(b\chi_B)(x)-b(x)\geq |m_{B}(b)|-b^+(x)+b^-(x),\quad x\in  B,$$
we arrive at
\begin{align*}
C&\geq {1\over \mu(B)}\int_B\left|2M^\sharp\left(b\chi_B \right)(x)-b(x) \right|d\mu(x)\\
&\geq {1\over\mu(B)}\int_B\left(2M^\sharp\left(b\chi_B \right)(x)-b(x) \right)d\mu(x)\\
&\geq {1\over\mu(B)}\int_B\left(|m_{B}(b)|-b^+(x)+b^-(x) \right)d\mu(x)\\
&=|m_{B}(b)|-{1\over\mu(B)}\int_Bb^+(x)d\mu(x)+{1\over\mu(B)}\int_Bb^-(x)d\mu(x).
\end{align*}

\noindent Letting $\mu(B)\to 0$ with $x\in B$, by Lebesgue differentiation theorem, we have
$$c\geq |b(x)|-b^+(x)+b^-(x)=2b^-(x).$$

\noindent The implication of (ii)$\Rightarrow$(iv) is similar to (iii)$\Rightarrow$(v).  This finishes the proof of Theorem \ref{thm-msharp}.
\end{proof}

\color{black}

\begin{proof}[Proof of Theorem \ref{thm-lp}]
 {\bf Sufficient condition}: Assume that $b\in {\rm BMO}(X)$, by Corollary \ref {cor-cm} and the fact that $M$ is bounded on $L^p(X)$, we can see that,  for every $f\in L^p(X)$, $1<p<\infty$,

 \begin{align*}
 \left\|C_b(f)\right\|_{L^p(X)}\leq C\|b\|_{{\rm BMO}(X)}\left\|M^2f\right\|_{L^p(X)}\leq C\|b\|_{{\rm BMO}(X)}\left\|f\right\|_{L^p(X)}.
 \end{align*}

\noindent { \bf Necessary condition}: Assume that $C_b$ is bounded on $L^p(X)$, next we will show that $b\in {\rm BMO}(X)$.
By H\"older's inequality, we have
\begin{align*}
\inf_{c\in\mathbb R}\int_{B}|b(x)-c|d\mu(x)&\leq \inf_{y\in B}\int_B|b(x)-b(y)|d\mu(x)\\
&\leq {1\over \mu(B)}\int_B\int_B |b(x)-b(y)|d\mu(x)d\mu(y)\\
&\leq\mu(B)^{1\over p'}\left[\int_B\left({1\over \mu(B)}\int_B|b(x)-b(y)|d\mu(y)\right)^pd\mu(x) \right]^{1\over p}.\end{align*}

\noindent Since for any $x\in B$,
$${1\over \mu(B)}\int_B|b(x)-b(y)|d\mu(y)\leq C_b(\chi_B)(x).$$
We can obtain that
\begin{align*}
\inf_{c\in\mathbb R}\int_{B}|b(x)-c|d\mu(x)&\leq \mu(B)^{1\over p'}\left(\int_B|C_b(\chi_B)(x)|^pd\mu(x) \right)^{1\over p}\\
&\leq \mu(B) \|C_b\|_{L^p(X)\to L^p(X)}.
\end{align*}

\noindent Therefore, $b\in {\rm BMO}(X)$, and
$$\|b\|_{{\rm BMO}(X)}\leq C\|C_b\|_{L^p(X)\to L^p(X)}.$$
This completes the proof of the theorem.
\end{proof}

\begin{proof}[Proof of Theorem \ref {thm-cb-iff}]
{\bf Sufficient condition:} Let $B_0$ be any fixed ball and $f=\chi_{B_0}$. For any $\lambda>0$, we have
\begin{align*}
\mu\left(\left\{ x\in X: C_b(f)(x)>\lambda \right\} \right)
&=\mu\Big(\Big\{ x\in X: \sup_{B\ni x}{1\over\mu(B)}\int_{B\cap B_0}|b(x)-b(y)| d\mu(y)>\lambda \Big\} \Big)\\
&\geq \mu\Big(\Big\{ x\in B_0: \sup_{B\ni x}{1\over\mu(B)}\int_{B\cap B_0}|b(x)-b(y)| d\mu(y)>\lambda \Big\}\Big)\\
&\geq \mu\Big(\Big\{ x\in B_0: {1\over\mu(B_0)}\int_{ B_0}|b(x)-b(y)| d\mu(y)>\lambda \Big\}\Big)\\
&\geq \mu\left(\left\{ x\in B_0: |b(x)-m_{B_0}(b)| >\lambda \right\}\right).
\end{align*}

\noindent By assumption, we have
\begin{align*}
\mu\left(\left\{ x\in B_0: |b(x)-m_{B_0}(b)| >\lambda \right\}\right)
&\leq C\int_X{|f(x)|\over\lambda}\Big(1+\log^+\Big({|f(x)|\over\lambda} \Big) \Big)d\mu(x)\\
&=C{1\over\lambda}\Big(1+\log^+{1\over\lambda} \Big)\mu(B_0).
\end{align*}
Then by Lemma \ref{lem-bmoq}, for $0<p<1$, we have
\begin{align*}
\int_{B_0}|b(x)-m_{B_0}(b)|^pd\mu(x)&=p\int_0^{\infty}\lambda^{p-1}\mu(\{x\in B_0:|b(x)-m_{B_0}(b)|>\lambda\})d\lambda\\
&=p\int_0^{1}\lambda^{p-1}\mu(\{x\in B_0:|b(x)-m_{B_0}(b)|>\lambda\})d\lambda\\
&\quad+p\int_1^{\infty}\lambda^{p-1}\mu(\{x\in B_0:|b(x)-m_{B_0}(b)|>\lambda\})d\lambda\\
&\leq p\mu(B_0)\int_0^{1}\lambda^{p-1}d\lambda+Cp\mu(B_0)\int_1^{\infty}\lambda^{p-1}
{1\over\lambda}\Big(1+\log^+{1\over\lambda} \Big)d\lambda\\
&=\mu(B_0)+Cp\mu(B_0)\int_1^{\infty}\lambda^{p-2}d\lambda\\
&=\Big(1+C{p\over 1-p}\Big)\mu(B_0).
\end{align*}
Therefore,
$$\Big({1\over\mu(B_0)}\int_{B_0}|b(x)-m_{B_0}(b)|^pd\mu(x)\Big)^{1\over p}\leq \Big(1+C{p\over 1-p}\Big)^{1\over p}.$$

\noindent Then it follows from Lemma \ref {lem-bmoq} that $b\in {\rm BMO}(X)$.\\

\noindent {\bf Necessary condition:} By Corollary \ref {cor-cm}, \eqref {mml} and  Lemma \ref {lem-ml}, we have
\begin{align*}
&\mu\left(\left\{ x\in X: C_b(f)(x)>\lambda \right\} \right)\\
&\leq\mu\Big(\Big\{ x\in X: M^2f(x)>{\lambda\over C\|b\|_{{\rm BMO}(X)}} \Big\} \Big)\\
&\leq\mu\Big(\Big\{ x\in X: M_{L\log L}f(x)>{\lambda\over C\|b\|_{{\rm BMO}(X)}} \Big\} \Big)\\
&\leq C\int_X {C\|b\|_{{\rm BMO}(X)}|f(y)|\over \lambda}\Big(1+\log^+\Big({C\|b\|_{{\rm BMO}(X)}|f(y)|\over \lambda} \Big) \Big)d\mu(y).
\end{align*}

\noindent Since for any $\alpha, \beta>0$,
\begin{align}\label{ab}
1+\log^+(\alpha\beta)\leq (1+\log^+\alpha)(1+\log^+\beta),
\end{align}

\noindent we have
\begin{align}\label{cbbmo}
\nonumber&\mu\left(\left\{ x\in X: C_b(f)(x)>\lambda \right\} \right)\\
&\leq C\|b\|_{{\rm BMO}(X)}\left(1+\log^+\|b\|_{{\rm BMO}(X)} \right)\int_X {|f(y)|\over \lambda}\Big(1+\log^+\Big({|f(y)|\over \lambda} \Big) \Big)d\mu(y).
\end{align}

This finishes the proof of Theorem \ref {thm-cb-iff}.
\end{proof}

\begin{proof}[Proof of Theorem \ref {thm-mb}]
By  Lemma \ref{lem-cbm}, we have
\begin{align*}
&\mu\left(\left\{x\in X: |[M,b]f(x)|>\lambda \right\}\right)\\
&\leq\mu\Big(\Big\{x\in X: C_b(f)(x)>{\lambda\over 2} \Big\}\Big)
+\mu\Big(\Big\{x\in X: 2b^-(x)Mf(x)>{\lambda\over 2} \Big\}\Big)\\
&\leq\mu\Big(\Big\{x\in X: C_b(f)(x)>{\lambda\over 2} \Big\}\Big)
+\mu\Big(\Big\{x\in X: 2\|b^-\|_{L^\infty(X)}Mf(x)>{\lambda\over 2} \Big\}\Big).
\end{align*}

By \eqref{ab} and \eqref{cbbmo}, we can see that
\begin{align*}
&\mu\Big(\Big\{x\in X: C_b(f)(x)>{\lambda\over 2} \Big\}\Big)\\
&\leq C\|b\|_{{\rm BMO}(X)}\left(1+\log^+\|b\|_{{\rm BMO}(X)} \right)\int_X {|f(y)|\over \lambda}\Big(1+\log^+\Big({|f(y)|\over \lambda} \Big) \Big)d\mu(y).
\end{align*}

Since $M$ is weak-type (1,1), we have
\begin{align*}
\mu\Big(\Big\{x\in X: 2\|b^-\|_{L^\infty(X)}Mf(x)>{\lambda\over 2} \Big\}\Big)
\leq C\|b^-\|_{L^\infty(X)}\int_X {|f(y)|\over \lambda}d\mu(y).
\end{align*}
Therefore,
\begin{align*}
&\mu\left(\left\{x\in X: |[M,b]f(x)|>\lambda \right\}\right)\\
&\leq C\left(\|b^+\|_{{\rm BMO}(X)}+\|b^-\|_{L^\infty} \right) \left(1+\log^+\|b^+\|_{{\rm BMO}(X)}\right)
\int_X{|f(x)|\over \lambda}\Big(1+\log^+\Big( {|f(x)|\over \lambda}\Big) \Big)dx.
\end{align*}

The proof of Theorem \ref {thm-mb} is complete.
\end{proof}

\begin{remark}
We now show that in the general setting of space of homogeneous type, $[M,b]$ fails to be of weak type $(1,1)$. We provide a counter example as follows. Assume that $(X,\rho,\mu)$ is a space of homogeneous type, where $\mu$ satisfies that,  for any $x_0\in X$, ${\log t\over \mu(B(x_0,t))}$ is decreasing on $(1,\infty)$. Let $b(x)=\log(1+d(x, x_0))\in {\rm BMO}(X)$ and let $f(x)=\chi_{B(x_0,1)}(x)$. Then for any $x\not\in B(x_0,2)$,
$$Mf(x)=\sup_{B\ni x}{1\over\mu(B)}\int_{B\cap B(x_0,1)}d\mu(y)=\sup_{B\ni x}{\mu(B\cap B(x_0,1))\over \mu(B)}.$$

\noindent So

$$b(x) M(f)(x) \geq \log(1+d(x, x_0)) \sup_{B\ni x}{\mu(B\cap B(x_0,1))\over \mu(B)}.  $$

\noindent On the other hand, for any $x\not\in B(x_0,2)$,

\begin{align*}
M(bf)(x)&=\sup_{B\ni x}{1\over\mu(B)}\int_{B\cap B(x_0,1)}\log(1+d(y, x_0))d\mu(y)
\leq \log2 \sup_{B\ni x}{\mu(B\cap B(x_0,1))\over \mu(B)}.
\end{align*}
Therefore, for any $x\not\in B(x_0,2)$,
\begin{align*}
|[M, b]f(x)|&=|M(bf)(x)-b(x)Mf(x)|\geq | \log(1+d(x, x_0)) -\log2 | \sup_{B\ni x}{\mu(B\cap B(x_0,1))\over \mu(B)}.
\end{align*}

Next, it is clear that for $x\not\in B(x_0,2)$,

$$ C_0 {\mu( B(x_0,1))\over \mu\big(x_0,{d(x,x_0)}\big)} \leq
{\mu( B(x_0,1))\over \mu(B_x)} \leq \sup_{B\ni x}{\mu(B\cap B(x_0,1))\over \mu(B)}\leq1,$$

\noindent where $B_x$ is the ball containing $x$ and $B(x_0,1)$. Hence, we see that for any $x\not\in B(x_0,100)$,
\begin{align*}
|[M, b]f(x)|&\geq {C_0\over2} \log(1+d(x, x_0))  {\mu( B(x_0,1))\over \mu\big(B(x_0,{d(x,x_0)})\big)} \geq C_1{\log(d(x, x_0)) \ \mu( B(x_0,1)) \over \mu\big(B(x_0,{d(x,x_0)})\big) }.
\end{align*}

\noindent Therefore, for any $\lambda>0$,
\begin{align*}
&\lambda \mu(\{ x\in X:  |[M, b]f(x)|>\lambda\}) \\
&\geq\lambda \mu\left(\left\{ x\in X\backslash B(x_0,100):  |[M, b]f(x)|>\lambda\right\}\right) \\
&\geq\lambda \mu\left(\left\{ x\in X\backslash B(x_0,100):  C_1{\log(d(x, x_0)) \ \mu( B(x_0,1)) \over \mu\big(B(x_0,{d(x,x_0)})\big) }>\lambda\right\}\right)\\
&=\lambda \mu\left(\left\{ x\in X\backslash B(x_0,100):  C_2{\log(d(x, x_0)) \over \mu\big(B(x_0,{d(x,x_0)})\big) }>\lambda\right\}\right).
\end{align*}

\noindent Let $\varphi(t)={\log t\over \mu(B(x_0,t))}$. Then it is a decreasing function on $(100,\infty)$, and
\begin{align*}
&\lambda \mu(\{ x\in X:  |[M, b]f(x)|>\lambda\}) \\
&\geq\lambda \mu\left(B\left(x_0,\varphi^{-1}\left({\lambda\over C_2}\right)\right)\right)-\lambda \mu ( B(x_0,100)).
\end{align*}

\noindent Therefore,
\begin{align*}
\lim_{\lambda\to 0}\lambda \mu\left(B\left(x_0,\varphi^{-1}\left({\lambda\over C_2}\right)\right)\right)
=\lim_{t\to \infty}C_2 \mu(B(x_0,t))\varphi(t)=\lim_{t\to \infty}\log t=\infty.
\end{align*}
\end{remark}

We provide a natural example of space of homogeneous type $(X,\rho,\mu)$ beyond the Euclidean setting, such that $\mu$ satisfies the assumption as in the remark above: for any $x_0\in X$, ${\log t\over \mu(B(x_0,t))}$ is decreasing on $(1,\infty)$.

\begin{example}
We recall the Bessel operator and its underlying space studied by Muckenhoupt and Stein $($\cite{MS}$)$.Consider $\R_+=(0,\infty)$. For $\lambda> -{1\over2}$, the Bessel operator $\Delta_\lambda$ on $\mathbb R_+$  is defined by
\begin{align*}
\Delta_\lambda = -{d^2\over dx^2} -{2\lambda\over x} {d\over dx}.
\end{align*}
It is a formally self-adjoint operator in $L^2(\mathbb R_+, dm_\lambda)$, where $d m_\lambda(x)=x^{2\lambda}dx$. It is clear that the corresponding underlying space $(\mathbb R_+, |\cdot|, dm_\lambda)$ is a space of homogeneous type in the sense of Coifman and Weiss.

For any $x\in \mathbb{R}_+$ and $r>0$, let $I(x, r)=(x-r, x+r)\cap \mathbb{R}_+$.
When $r\le x$, by mean value theorem, we have
$$m_\lambda(I(x, r))=\int_{x-r}^{x+r} y^{2\lambda} dy={1\over 2\lambda+1}\left[(x+r)^{2\lambda+1}-(x-r)^{2\lambda+1}\right]=2(x+\theta r)^{2\lambda}r,$$
for some $0<\theta<1$.
 When $r>x$,
$$m_\lambda(I(x, r))=\int_{0}^{x+r} y^{2\lambda} dy={1\over 2\lambda+1}(x+r)^{2\lambda+1}.$$

\noindent Therefore, we can see that ${\log r\over m_\lambda(I(x, r))}$ is decreasing on $(1,\infty)$.
\end{example}

\medskip
\section{Weighted version of the commutator theorems}
\label{sec:weighted}

We will further provide the weighted version of Theorems \ref{thm-mp}, \ref{thm-msharp} and \ref{thm-lp} by
establishing new characterisations of ${\rm BMO}(X)$ via having the Muckenhoupt weight in both the denominator and the integrand in the definition of BMO norm.

We recall that a locally integrable function $w:X\rightarrow (0,+\infty)$ is an ${\mathcal{A}}_p$ weight (or $w\in{\mathcal{A}}_p$), $1<p<\infty$ if

$$\left[w\right]_{{\mathcal{A}}_p}:=\sup_{B} \left(\frac{1}{\mu(B)}\int_B w(x) d\mu(x)\right)\left(\frac{1}{\mu(B)}\int_B w(x)^{-1/(p-1)} d\mu(x)\right)^{p-1}<\infty.$$

\noindent A function $w(x)\geq 0$ is called an ${\mathcal{A}}_1$ weight (or $w\in{\mathcal{A}}_1$) if $M(w)(x)\leq C w(x)$ for $x\in X$. The class of ${\mathcal{A}}_p$ weights is increasing with $p$ for $1\leq p<\infty$. The next lemma is a weighted version of Lemma \ref{lem-mpb}.

\begin{lemma}\label{lem-mpb-weight}
Let $b\in{\rm BMO}(X)$ be a nonnegative function, $w$ be an ${\mathcal{A}}_q$ weight and $1<q<\infty$. Suppose that $T$ is a quasilinear operator satisfying (a)-(c) from Section \ref{sec:pointwise} and is bounded on $L^q(X,wd\mu)$. Then
$[T, b]$ is bounded on $L^q(X,wd\mu)$.
\end{lemma}

\begin{proof}
We use a similar argument to \cite[Theorem 4.4]{MiS}. Let $d$ be a real constant and $m(x)=e^{db(x)}$ for $x\in X$. We claim that if $\vert d\vert\leq C_3/\delta'$ then $mw\in {\mathcal{A}}_q$. Let $\delta>1$ be such that $w^{\delta}\in {\mathcal{A}}_q$,  $\alpha_1(B)=\frac{1}{\mu(B)}\int_{B}w(x)e^{db(x)}d\mu(x)$ and $$\alpha_2(B)=\left(\frac{1}{\mu(B)}\int_{B} w(x)^{-1/(q-1)} e^{-db(x)/(q-1)}d\mu(x)\right)^{q-1}, $$for any ball $B\subset X$. Using Holder's inequality with $\delta>1$ and its conjugate $\delta'$ ($\frac{1}{\delta}+\frac{1}{\delta'}=1$) in $\alpha_1(B)$ and $\alpha_2(B)$, we get for any ball $B\subset X$,

$$\alpha_1(B)\leq \left(\frac{1}{\mu(B)}\int_{B}w(x)^{\delta}d\mu(x)\right)^{1/\delta}\left(\frac{1}{\mu(B)}\int_{B}e^{db(x)\delta'}d\mu(x)\right)^{1/\delta'} $$

\noindent and

$$\alpha_2(B)\leq\left(\frac{1}{\mu(B)}\int_{B} w(x)^{-\delta/(q-1)}d\mu(x)\right)^{(q-1)/\delta}\left(\frac{1}{\mu(B)}\int_B e^{-db(x)\delta'/(q-1)}d\mu(x)\right)^{(q-1)/\delta'}. $$

Let $m_B(b)=\frac{1}{\mu(B)}\int_Bb(x) d\mu(x)$. Observe that

\begin{align*}
&\left(\frac{1}{\mu(B)}\int_{B}e^{db(x)\delta'}d\mu(x) \right)\left(\frac{1}{\mu(B)}\int_B e^{-db(x)\delta'/(q-1)}d\mu(x)\right)^{(q-1)}\\
&\leq\frac{1}{\mu(B)}\int_{B}e^{d(b(x)-m_B(b))\delta'}d\mu(x) \left(\frac{1}{\mu(B)}\int_B e^{-d(b(x)-m_B(b))\delta'/(q-1)}d\mu(x)\right)^{(q-1)}\\
&\leq\frac{1}{\mu(B)}\int_{B}e^{\vert d\vert\cdot\vert b(x)-m_B(b)\vert\delta'}d\mu(x) \left(\frac{1}{\mu(B)}\int_B e^{\vert d\vert \cdot\vert b(x)-m_B(b)\vert\delta'/(q-1)}d\mu(x)\right)^{(q-1)}.
\end{align*}

\noindent Using Lemma \ref{lem-jn2}, we obtain that

$$\sup_{B}\frac{1}{\mu(B)}\int_{B}e^{\vert d\vert\cdot\vert b(x)-m_B(b)\vert\delta'}d\mu(x)\leq C_4,$$

\noindent for $\vert d\vert\leq C_3/\delta'$. If $q\geq 2$ then,
\begin{align*}
\sup_{B}\left(\frac{1}{\mu(B)}\int_B e^{\vert d\vert \cdot\vert b(x)-m_B(b)\vert\delta'/(q-1)}d\mu(x)\right)^{(q-1)}&\leq \sup_{B}\left(\frac{1}{\mu(B)}\int_B e^{\vert d\vert \cdot\vert b(x)-m_B(b)\vert\delta'}d\mu(x)\right) \\
&\leq C_4.
\end{align*}

\noindent The same inequality holds for $1<q<2$ with $\vert d\vert\leq C_3/\delta'$, since for $q'=\frac{q}{q-1}$ the conjugate of $q$, we have

$$\left[e^{\vert d\vert \cdot\vert b(x)-m_B(b)\vert\delta'}\right]_{{\mathcal{A}_q}}= \left[e^{-\vert d\vert \cdot\vert b(x)-m_B(b)\vert\delta'/(q'-1)}\right]_{{\mathcal{A}_{q'}}}^{q'-1}.$$
%

\noindent  Now, as $w^{\delta}\in {\mathcal{A}}_q$, we obtain that

$$\sup_{B}\left(\frac{1}{\mu(B)}\int_{B}w(x)^{\delta}d\mu(x)\right)^{1/\delta}\left(\frac{1}{\mu(B)}\int_{B} w(x)^{-\delta/(q-1)}\right)^{(q-1)/\delta}<\infty.$$

\noindent It follows that
$$\sup_B\alpha_1(B)\alpha_2(B)\leq C_5$$

\noindent Choose $d$ with $\vert d\vert\leq C_3/\delta'$. We apply \cite[Theorem 4.4]{MiS} with the pair of Banach spaces $\overline{A}=\left(L^p(X,e^{-db}\omega d\mu),L^p(X,e^{db}\omega d\mu)\right)$ using that $\overline{A}_{1/2,q}=L^q(X,wdx)$ to obtain that $[T, b]$ is bounded on $L^q(X,wdx)$.
The proof of Lemma \ref{lem-mpb-weight} is complete.
\end{proof}

The same result holds if $b\in {\rm BMO}(X)$ with $b^{-}\in L^{\infty}(X)$ using the inequality $$\left\vert [T,b]f-[T,\vert b\vert]f\right\vert\leq 2(b^{-}T(f)+T(b^{-}f))$$ for $f\in L^p(X,wd\mu)$.

\begin{theorem}\label{thm-mp-weight}
Let $b$ be a real valued, locally integrable function in $X$ and $w$ an ${\mathcal{A}}_p$ weight for $1<p<\infty$. The following assertions are equivalent:

{\rm (i)} $b\in {\rm BMO}(X)$ and $b^-\in L^\infty(X)$;

{\rm (ii)} The commutator $[M_p, b]$ is bounded on $L^q(X,wd\mu)$, for all $q$, $p<q<\infty$;

{\rm (iii)} The commutator $[M_p, b]$ is bounded on $L^q(X,wd\mu)$, for some $q$, $p<q<\infty$.

{\rm (iv)} For all $q\in [1,\infty)$, we have
$$\sup_{B}{1\over w(B)}\int_B|b(x)-M_p(b\chi_B)(x)|^q w(x)d\mu(x)<\infty;$$

{\rm (v)} There exists $q\in [1,\infty)$ such that
$$\sup_{B}{1\over w(B)}\int_B|b(x)-M_p(b\chi_B)(x)|^q w(x)d\mu(x)<\infty.$$

\end{theorem}

\begin{proof}[Proof of Theorem \ref{thm-mp-weight}]
It is clear that (ii) implies (iii). For (i)$\Rightarrow$(ii), since $M_p$ satisfies (a)-(c) and is bounded on $L^q(X,wd\mu)$ for $p<q<\infty$, the result follows from Lemma \ref{lem-mpb-weight}. For (iii)$\Rightarrow$(v), as in the unweighted case, one can write $(b\chi_B-M_{p,B}(b))\chi_B=(bM_p(\chi_B)-M_p(b\chi_B))\chi_B=[M_p,b](\chi_B)\chi_B$ and

\begin{align*}
\int_{B}\vert b-M_{p,B}(b)\vert^q w d\mu&=\int_B\vert [M_p,b](\chi_B)\vert^qwd\mu\\
&\leq\Vert [M_p,b](\chi_B)\Vert_{L^q(X,wd\mu)}^q\leq C\Vert \chi_B\Vert_{L^q(X,wd\mu)}^q,
\end{align*}

\noindent where the last inequality follows from the boundedness of $[M_p,b]$ on $L^q(X,wd\mu)$. As $\Vert \chi_B\Vert_{L^q(X,wd\mu)}^q=\nu(B)$, we obtain (v). For (v)$\Rightarrow$(i), let $B$ be a fixed ball. By H\"older's inequality, we have

\begin{align*}
{1\over \mu(B)}\int_B|b-M_{p,B}(b)|d\mu(x)&={1\over \mu(B)}\int_B|b-M_{p,B}(b)|w^{1/q}(x) w(x)^{-1/q} d\mu(x)\\
&\leq \left({1\over \mu(B)}\int_B\left|b-M_{p,B}(b)\right|^q w(x)d\mu(x) \right)^{1\over q}\\
&\qquad\qquad\qquad\qquad\times\left({1\over\mu(B)}\int_B w(x)^{-q'/q} d\mu(x)\right)^{1/q'}\\
&\leq K\left({1\over \mu(B)}\int_B\left |b-M_{p,B}(b)\right|^q w(x)d\mu(x) \right)^{1\over q}\\\
&\qquad\qquad\qquad\qquad\times\left({1\over\mu(B)}\int_B w(x)d\mu(x)\right)^{-1/q}\\
&\leq K \left({1\over w(B)}\int_B\left |b-M_{p,B}(b)\right|^q w(x)d\mu(x) \right)^{1\over q}\\
&\leq C.
\end{align*}

\noindent Now, in the same way as the unweighted version, we have
\begin{align*}
{1\over \mu(B)}\int_B|b-m_{B}(b)|d\mu&=\int_{E_1}|b-M_{p,B}(b)|d\mu\\
&\leq{2\over \mu(B)}\int_B|b-M_{p,B}(b)|d\mu\leq C.
\end{align*}

\noindent Therefore, $b\in{\rm BMO}(X)$. In the same way as the unweighted case, $b^-\in L^{\infty}(X)$. The implication of (ii)$\Rightarrow$(iv) is similar to (iii)$\Rightarrow$(v). This completes the proof of Theorem \ref {thm-mp-weight}.
\end{proof}

\begin{theorem}\label{thm-msharp-weight}
Let $b$ be a real valued, locally integrable function in $X$ and $w$ be an ${\mathcal{A}}_1$ weight. The following assertions are equivalent:

{\rm (i)} $b\in {\rm BMO}(X)$ and $b^-\in L^\infty(X)$;

{\rm (ii)} The commutator $[M^\sharp, b]$ is bounded on $L^q(X,wd\mu)$, for all $q$, $1<q<\infty$;

{\rm (iii)} The commutator $[M^\sharp, b]$ is bounded on $L^q(X,wd\mu)$, for some $q$, $1<q<\infty$;

{\rm (iv)} For all $q\in [1,\infty)$, we have
$$\sup_{B}{1\over w(B)}\int_B|b(x)-2M^\sharp(b\chi_B)(x)|^q w(x)d\mu(x)<\infty;$$

{\rm (v)} There exists $q\in [1,\infty)$ such that
$$\sup_{B}{1\over w(B)}\int_B|b(x)-2M^\sharp(b\chi_B)(x)|^q w(x)d\mu(x)<\infty.$$
\end{theorem}

\begin{proof}[Proof of Theorem \ref {thm-msharp-weight}]
\noindent Obviously, (ii)$\Rightarrow$(iii), and (iv)$\Rightarrow$(v).

 For (i)$\Rightarrow$(ii), since $M^\sharp$ satisfies (a)-(c) and is bounded on $L^q(X,wd\mu)$, the result follows from Lemma \ref{lem-mpb-weight}.

 (iii)$\Rightarrow$(v). Let $B$ be a fixed ball. We know from the Proof of Theorem \ref {thm-msharp} that $M^\sharp\left(\chi_B \right)(x)={1\over 2}$, for all $x\in B$ and ${1\over 2}b\chi_B-M^\sharp(b\chi_B)\chi_B=[M^\sharp,b](\chi_B)\cdot \chi_B$. By the assumption of (iii), one has
\begin{align*}
&\left(\int_B\left\vert {1\over 2}b(x)-M^\sharp(b\chi_B)(x)\right\vert^q\ w(x)d\mu(x) \right)^{1/q}\\
&= \Vert[M^\sharp,b](\chi_B)\chi_B\Vert_{L^q(X,wd\mu)}\leq  \Vert[M^\sharp,b](\chi_B)\Vert_{L^q(X,wd\mu)}\\
&\leq C \Vert \chi_B\Vert_{L^q(X,wd\mu)}\\
&\leq C w(B)^{1/q},
\end{align*}

\noindent which implies (iv). For (v)$\Rightarrow$(i), from the Proof of Theorem \ref{thm-msharp}, we have
$$|m_{B}(b)|\leq 2M^\sharp\left( b\chi_B\right) (x), \quad x\in B.$$

\noindent To show that $b\in {\rm BMO}(X)$, we use the same arguments as in the Proof of Theorem \ref{thm-msharp} and Theorem \ref{thm-mp-weight} with the set $E=\{x\in B:b(x)\leq m_{B}(b)\}$. For ${1\over q}+{1\over q'}=1$, we get

\begin{align*}
&{1\over \mu(B)}\int_B\left| b(x)-m_{B}(b)\right|d\mu(x)\\
&={2\over \mu(B)}\int_E\left( b(x)-m_{B}(b)\right)d\mu(x)\\
&\leq{2\over \mu(B)}\int_E\left (2 M^\sharp\left( b\chi_B\right)(x)-b(x)\right)d\mu(x)\\
&\leq{2\over \mu(B)}\int_B\left| 2 M^\sharp\left( b\chi_B\right)(x)-b(x)\right| w(x)^{1/q} w(x)^{-1/q}d\mu(x)\\
&\leq 2\left({1\over \mu(B)}\int_B\left| 2 M^\sharp\left( b\chi_B\right)(x)-b(x)\right|^q w(x)d\mu(x)\right)^{1/q}\left({1\over \mu(B)}\int_B w(x)^{-q'/q}d\mu(x)\right)^{1/q'}\\
&\leq 2\left({1\over \mu(B)}\int_B\left| 2 M^\sharp\left( b\chi_B\right)(x)-b(x)\right|^q w(x)d\mu(x)\right)^{1/q}{[w]_{{\mathcal{A}}_q}^{1/q}\over \mu(B)^{-1/q}}\left(\int_B w(x) d\mu(x)\right)^{-1/q}\\
&\leq K \left({1\over w(B)}\int_B\left| 2 M^\sharp\left( b\chi_B\right)(x)-b(x)\right|^q w(x)d\mu(x)\right)^{1/q}\\
&\leq C.
\end{align*}

\noindent One can prove in the same way as in the Proof of Theorem \ref{thm-msharp} that $b^-\in L^{\infty}(X)$. The implication of (ii)$\Rightarrow$(iv) is similar to (iii)$\Rightarrow$(v), which ends the proof of Theorem \ref{thm-msharp-weight}.
\end{proof}

\color{black}

\begin{theorem}\label{thm-lp-weight}
Let $b\in L_{loc}^1(X)$, $1<p<\infty$ and $w\in {\mathcal{A}}_p$. Then the maximal commutator $C_b$ is bounded on $L^p(X,wd\mu)$ if and only if $b\in {\rm BMO(X)}$.
\end{theorem}

\begin{proof}[Proof of Theorem \ref{thm-lp-weight}]
{\bf{Sufficient condition}}: It follows the same argument as in the unweighted case. If $b\in {\rm BMO}(X)$, by Corollary \ref {cor-cm} and the boundedness of $M$ on $L^p(X,wd\mu)$, we obtain that, for every $f\in L^p(X,wd\mu)$, $1<p<\infty$,
 \begin{align*}
 \left\|C_b(f)\right\|_{L^p(X,wd\mu)}\leq C\|b\|_{{\rm BMO}(X)}\left\|M^2f\right\|_{L^p(X,wd\mu)}\leq C\|b\|_{{\rm BMO}(X)}\left\|f\right\|_{L^p(X,wd\mu)}.
 \end{align*}

\noindent {{\bf Necessary condition}}: Assume that $C_b$ is bounded on $L^p(X,wd\mu)$. By H\"older's inequality, we get
\begin{align*}
&\inf_{c\in\mathbb R}\int_{B}|b(x)-c|d\mu(x)\\
&\leq \inf_{y\in B}\int_B|b(x)-b(y)|d\mu(x)\\
&\leq {1\over \mu(B)}\int_B\int_B |b(x)-b(y)|d\mu(x)d\mu(y)\\
&\leq {\left[w\right]_{{\mathcal{A}}_p}}^{1/p}w(B)^{-{1\over p}}\mu(B)\left[\int_B\left({1\over \mu(B)}\int_B|b(x)-b(y)|d\mu(y)\right)^p w(x)d\mu(x) \right]^{1\over p}
.\end{align*}

\noindent We use that for any $x\in B$, ${1\over \mu(B)}\int_B|b(x)-b(y)|d\mu(y)\leq C_b(\chi_B)(x)$ to obtain that

\begin{align*}
{1\over\mu(B)}\inf_{c\in\mathbb R}\int_{B}|b(x)-c|d\mu(x)&\leq w(B)^{-{1\over p}}\left(\int_B|C_b(\chi_B)(x)|^pw(x)d\mu(x) \right)^{1\over p}\\
&\leq  {\left[w\right]_{{\mathcal{A}}_p}}^{1/p}w(B)^{-{1\over p}} \|C_b\|_{L^p(X,wd\mu)\to L^p(X,wd\mu)}\Vert \chi_B\Vert_{L^p(X,wd\mu)}\\
&\leq {\left[w\right]_{{\mathcal{A}}_p}}^{1/p} \|C_b\|_{L^p(X,wd\mu)\to L^p(X,wd\mu)}
\end{align*}

\noindent Therefore, $b\in {\rm BMO}(X)$, and
$$\|b\|_{{\rm BMO}(X)}\leq C\|C_b\|_{L^p(X,wd\mu)\to L^p(X,wd\mu)}.$$
This completes the proof of the theorem.
\end{proof}

\section{Local version of the commutator theorems}
\label{sec:local}
\setcounter{equation}{0}

We will sketch the result in the case that $\mu(X)<\infty$ in this section. Note that in this setting, the BMO space  ${\rm BMO }(X)$ is defined as the set of all $b\in L^1(X)$ such that
$$ \sup_{ x\in X, r>0, B(x,r)\subset X} {1\over\mu(B(x,r))}\int_{B(x,r)} |b(y)-m_B(b)|d\mu(y)<\infty, $$

\noindent and the norm is defined as
$$\|b\|_{{\rm BMO }(X)}:=  \sup_{ x\in X, r>0, B(x,r)\subset X} {1\over\mu(B(x,r))}\int_{B(x,r)} |b(y)-m_B(b)|d\mu(y)+ \|b\|_{L^1(X)}. $$

\begin{theorem}\label{thm local}
Suppose  $\mu(X)<\infty$, $\diam(X)<\infty$ and ${\rm BMO }(X)$ is defined as above.
Then the results in Theorems \ref{thm-mp}, \ref{thm-msharp}, \ref{thm-mb}, \ref{thm-lp} and
\ref{thm-cb-iff}  hold in this setting.
\end{theorem}

\begin{proof}[Proof of Theorem \ref{thm local}]
The proof is similar to those of Theorems \ref{thm-mp}, \ref{thm-msharp}, \ref{thm-mb}, \ref{thm-lp} and
\ref{thm-cb-iff}, except the case when proving $b\in {\rm BMO}(X)$ in these theorems. By assumption,  there exists $R_0>0$ such that for any $B(x,r)\subset X$, we have $r<R_0$. We test the ${\rm BMO}(X)$ condition on the case of balls with big radius and small radius. In the case of balls with small radius, $r<R_0$, the proof is the same as in Theorems \ref{thm-mp}, \ref{thm-msharp}, \ref{thm-mb}, \ref{thm-lp} and
\ref{thm-cb-iff}. In the case of balls  with large radius, $r\geq R_0$. By \eqref {upper dimension}, we obtain that
\begin{align*}
{1\over \mu(B)} \int_B|b(x)-m_B(b)|d\mu(x)&\leq 2 C_\mu\Big({\diam(X)\over r}\Big)^n{1\over \mu(X)} \|b\|_{L^1(X)}\\
&\leq 2 C_\mu\Big({\diam(X)\over R_0}\Big)^n{1\over \mu(X)} \|b\|_{L^1(X)}.
\end{align*}

\noindent This finishes the proof of Theorem \ref{thm local}.
\end{proof}

We remark that a concrete example of the space of homogeneous type $(X,d,\mu)$
with $\mu(X)<\infty$ and $\diam(X)<\infty$ is the boundary of a bounded strictly pseudoconvex domain in $\C^n$, see for example the recent works in \cite{DLLWW}, \cite{KL2} and \cite{LS}. To be more precise, we recall the bounded domain $D$  from  \cite{LS} with defining function $\rho$, which means that $D=\{z\in\mathbb C^n: \rho(z)<0\}$ with $\rho: \mathbb C^n\rightarrow \mathbb R$ and boundary $bD$. Without lost of generality, assume that $\rho$ is strictly plurisubharmonic (see \cite[Ch. II Sec. 4]{Ra}). Let $\mathcal L_0(w, z)$ be the negative of the Levi polynomial at $w\in bD$, given by

$$ \mathcal L_0(w, z) = \langle \partial\rho(w),w-z\rangle -{1\over2} \sum_{j,k} {\partial^2\rho(w) \over \partial w_j \partial w_k} (w_j-z_j)(w_k-z_k),  $$

\noindent where $\partial \rho(w)=\left({\partial\rho\over\partial w_1}(w),\cdots, {\partial\rho\over\partial w_n}(w)\right)$ and we have used the notation $\langle\eta,\zeta\rangle=\sum_{j=1}^n\eta_j\zeta_j$ for $\eta=(\eta_1,\cdots, \eta_n), \zeta=(\zeta_1,\dots,\zeta_n)\in\mathbb C^n$. The strict plurisubharmonicity of $\rho $ implies that

$$  2\operatorname{ Re} \mathcal L_0(w, z) \geq -\rho(z)+c |w-z|^2,  $$

\noindent for some $c>0$, whenever $w\in bD$ and $z\in \bar D$ is sufficiently close to $w$. Then a modification of $\mathcal L_0$ is as follows
\begin{align}\label{g0}
    g_0(w,z) = \chi \mathcal L_0+ (1-\chi) |w-z|^2.
\end{align}

\noindent Here $\chi=\chi(w,z)$ is a $C^\infty$-cutoff function with $\chi=0$ when $|w-z|\leq \mu/2$ and $\chi=1$ if $|w-z|\geq \mu$. Then for $\mu$ chosen sufficiently small (and then kept fixed throughout), we have that $ \operatorname{ Re}g_0(w,z)\geq c(-\rho(z)+ |w-z|^2)$ for $z$ in $\bar D$ and $w$ in $bD$, with $c$  a positive constant.

Note that the modified Levi polynomial $g_0$ has no smoothness beyond continuity in the variable $w$. So in \cite{LS}, for each $\epsilon>0$ the authors considered a variant $g_\epsilon$ defined as follows: let $\{\tau_{jk}^\epsilon(w)\}$ be an $n\times n$-matrix  of $C^1$ functions such that

$$\sup_{w\in bD}\Big|{\partial^2\rho(w)\over\partial w_j\partial w_k}- \tau_{jk}^\epsilon(w)\Big|\leq\epsilon,\quad 1\leq j,k\leq n.$$

\noindent Set

$$ \mathcal L_\epsilon(w, z) = \langle \partial\rho(w),w-z\rangle -{1\over2} \sum_{j,k}\tau_{jk}^\epsilon(w) (w_j-z_j)(w_k-z_k),  $$

\noindent and define

$$    g_\epsilon(w,z) = \chi \mathcal L_\epsilon+ (1-\chi) |w-z|^2, \quad z,w\in\mathbb C^n.  $$

\noindent Now $g_\epsilon$ is $C^1$ in $w$ and $C^\infty$ in $z$. We note that

$$\left|g_0(w,z)-g_\epsilon(w,z) \right|\lesssim \epsilon |w-z|^2.$$

\noindent We shall always assume that $\epsilon$ is sufficiently small, we then have

$$\left|g_\epsilon(w,z)\right|\approx\left|g_0(w,z) \right|,$$

\noindent where the equivalence $\approx$ is independent of $\epsilon$. Now on the boundary $bD$, define the function
${\tt d}(w,z) = |g_0(w,z)|^{1\over2}$. According to \cite[Proposition 3]{LS},  ${\tt d}$ satisfies that for all $w,z,z'\in bD$,

(a)  ${\tt d}(w,z)=0 $ iff $w=z$;

(b)  ${\tt d}(w,z)\approx {\tt d}(z,w)$;

(c)  ${\tt d}(w,z)\lesssim {\tt d}(w,z') +{\tt d}(z',z)$.

Next we recall the Leray--Levi measure $d\lambda$ on $bD$ defined via the $(2n-1)$-form
$$ {1\over (2\pi i)^n} \partial \rho \wedge (\bar\partial \partial \rho)^{n-1}. $$
To be more precise, we have the linear functional
\begin{align}\label{lambda}
f\mapsto {1\over (2\pi i)^n} \int_{bD} f(w) j^*(\partial \rho \wedge (\bar\partial \partial \rho)^{n-1})(w)=: \int_{bD} f(w)d\lambda(w)
\end{align}
defined for $f\in C(bD)$, and this defines the measure $d\lambda$.
Then one also has
$$   d\lambda(w) = {1\over (2\pi i)^n} j^*(\partial\rho \wedge (\bar\partial \partial\rho)^{n-1}) (w) =\Lambda(w)d\sigma(w),
 $$
where $j^*$ denotes the pullback under the inclusion $j:bD\hookrightarrow\mathbb C^n$, $d\sigma$ is the induced Lebesgue measure on $bD$ and $\Lambda(w)$ is a continuous function such that
$ c\leq \Lambda(w)\leq \tilde c, w\in bD$, with $c$ and $\tilde c$ two positive constants.

\begin{example}
Let $ (bD, {\tt d}, \lambda) $ be defined as above. Then it is a specific space of homogeneous type with $\lambda(bD)<\infty$ and $\diam(bD)<\infty$.

We also recollect the boundary balls $ B_r(w) $ determined via the quasidistance ${\tt d }$ and their measures, i.e.,
\begin{align}\label{ball}
B_r(w) =\{ z\in bD:\ {\tt d}(w,z)<r \}, \quad {\rm where\ } w\in bD.
\end{align}
According to \cite[p. 139]{LS},
\begin{align}\label{lambdab}
c_\lambda^{-1} r^{2n}\leq \lambda\big(B_r(w) \big)\leq c_\lambda r^{2n},\quad 0<r\leq 1,
\end{align}
for some $c_\lambda>1$.
\end{example}







\bigskip
\noindent \textbf{Acknowledgement:} This work was supported by Natural Science Foundation
of China (Grant Nos. 11671185, 11701250 and 11771195) and Natural Science Foundation of Shandong
Province (Grant Nos. ZR2018LA002 and ZR2019YQ04).

\bigskip



\bigskip
\bigskip
\bigskip

\bibliographystyle{amsplain}

\end{document}